\documentclass{amsart}
\usepackage{mathptmx}
\usepackage{amssymb}
\usepackage{amsfonts}
\usepackage{amsmath}
\usepackage{graphicx}
\usepackage{shadow}
\usepackage[all]{xy}
\usepackage[pagebackref]{hyperref}
\newtheorem{thm}{Theorem}[section]
\newtheorem{corollary}[thm]{Corollary}
\newtheorem{lemma}[thm]{Lemma}
\newtheorem{proposition}[thm]{Proposition}

\theoremstyle{definition}

\newtheorem{example}[thm]{Example}

\theoremstyle{remark}
\newtheorem{remark}[thm]{Remark}
\newtheorem{claim}{Claim}
\newcommand{\field}[1]{\mathbb{#1}}

\newcommand{\Z }{\field{Z}}
\newcommand{\N }{\field{N}}

\DeclareMathOperator{\htt}{ht}

\DeclareMathOperator{\Spec}{Spec}
\DeclareMathOperator{\Max}{Max}

\DeclareMathOperator{\gr}{gr}
\DeclareMathOperator{\G}{G}

\begin{document}

\title[On the prime ideal structure of symbolic Rees algebras]{On the prime ideal structure of symbolic Rees algebras}

\author{S. Bouchiba}
\address{Department of Mathematics, University of Meknes, Meknes 50000, Morocco}
\email{bouchiba@caramail.com}

\author{S. Kabbaj}
\address{Department of Mathematics and Statistics, KFUPM, Dhahran 31261, KSA}
\email{kabbaj@kfupm.edu.sa}
\thanks{Supported by KFUPM under Research Grant \# MS/DOMAIN/369.}

\date{\today}

\begin{abstract}
This paper contributes to the study of the prime spectrum and dimension theory of symbolic Rees algebra over Noetherian domains. We first establish some general results on the prime ideal structure of subalgebras of affine domains, which actually arise, in the Noetherian context, as domains between a domain $A$ and $A[a^{-1}]$. We then examine closely the special context of symbolic Rees algebras (which yielded the first counter-example to the Zariski-Hilbert problem). One of the results states that if $A$ is a Noetherian domain and $p$ a maximal ideal of $A$, then the Rees algebra of $p$ inherits the Noetherian-like behavior of being a stably strong S-domain. We also investigate graded rings associated with symbolic Rees algebras of prime ideals $p$ such that $A_{p}$ is a rank-one DVR and close with an application related to Hochster's result on the coincidence of the ordinary and symbolic powers of a prime ideal.
\end{abstract}
\maketitle

\begin{section}{Introduction}
All rings considered in this paper are integral domains and all ring homomorphisms are unital. Examples of finite-dimensional non-Noetherian Krull (or factorial) domains are scarce in the literature. One of these stems from the generalized fourteenth problem of Hilbert (also called Zariski-Hilbert problem). Let $k$ be a field of characteristic zero and let $T$ be a normal affine domain over $k$. Let $F$ be a subfield of the field of fractions of $T$. Set $R:=F\cap T$. The Hilbert-Zariski problem asks whether $R$ is an affine domain over $k$. Counterexamples on this problem were constructed by Rees \cite{Re}, Nagata \cite{Na1} and Roberts \cite{R1,R2}. In 1958, Rees constructed the first counter-example giving rise to (what is now called) Rees algebras. In 1970, based on Rees' work, Eakin and Heinzer constructed in \cite{EH} a first example of a 3-dimensional non-Noetherian Krull domain which arose as a symbolic Rees algebra. In 1973, Hochster studied in \cite{H} criteria for the ordinary and symbolic powers of a prime ideal to coincide (i.e., the Rees and symbolic Rees algebras are equal) within Noetherian contexts. Since then, these special graded algebras has been capturing the interest of many commutative algebraists and geometers.

In this line, Anderson, Dobbs, Eakin, and Heinzer \cite{ADEH} asked whether $R$ and its localizations inherit from $T$ the Noetherian-like main behavior of having Krull and valuative dimensions coincide  (i.e., $R$ is a Jaffard domain). This can be viewed in the larger context of Bouvier's conjecture about whether finite-dimensional non-Noetherian Krull domains are Jaffard \cite{BK,FK}. In \cite{BK2}, we showed that while most examples existing in the literature are (locally) Jaffard, the question about those arising as symbolic Rees algebras is still open. This lies behind our motivation to contribute to the study of the prime ideal structure of this construction. We examine contexts where it inherits the (locally) Jaffard property and hence compute its Krull and valuative dimensions.

A finite-dimensional domain $R$ is said to be Jaffard if $\dim(R[X_{1}, \cdots , X_{n}])= n + \dim(R)$ for all $n\geq 1$ or, equivalently, if $\dim(R) = \dim_{v}(R)$, where $\dim(R)$ denotes the (Krull) dimension of $R$ and $\dim_{v}(R)$ its valuative dimension (i.e., the supremum of dimensions of the valuation overrings of $R$). As this notion does not carry over to localizations, $R$ is said to be locally Jaffard if $R_p$ is a Jaffard domain for each prime ideal $p$ of $R$ (equiv., $S^{-1}R$ is a Jaffard domain for each multiplicative subset $S$ of $R$).

In order to study Noetherian domains and Pr\"{u}fer domains in a unified manner, Kaplansky \cite{Ka} introduced the
notions of S-domain and strong S-ring. A domain $R$ is called an S-domain if, for each height-one prime ideal $p$ of $R$, the extension $p[X]$ to the polynomial ring in one variable also has height $1$. A ring $R$ is said to be a strong S-ring if $\frac{R}{p}$ is an S-domain for each $p\in \Spec(R)$. While $R[X]$ is always an S-domain for any domain $R$ \cite{FK90}, $R[X]$ need not be a strong S-ring even when $R$ is a strong S-ring. Thus, $R$ is said to be a stably  (or universally) strong S-ring if the polynomial ring $R[X_1,\cdots ,X_n]$ is a strong S-ring for each positive integer $n$ \cite{K1,K2,MM}. A stably  strong S-domain is locally Jaffard \cite{ABDFK,K1}. An example of a strong S-domain which is not a stably strong S-domain was constructed in \cite{BMRH}. We assume familiarity with these concepts, as in \cite{ABDFK,ADKM,BDF,BK,DFK,J,K1,K2,MM}.

In Figure~\ref{D}, a diagram of implications indicates how the classes of Noetherian domains, Pr\"ufer domains, UFDs, Krull domains, and PVMDs \cite{G1} interact with the notion of Jaffard domain as well as with the S-properties.\bigskip

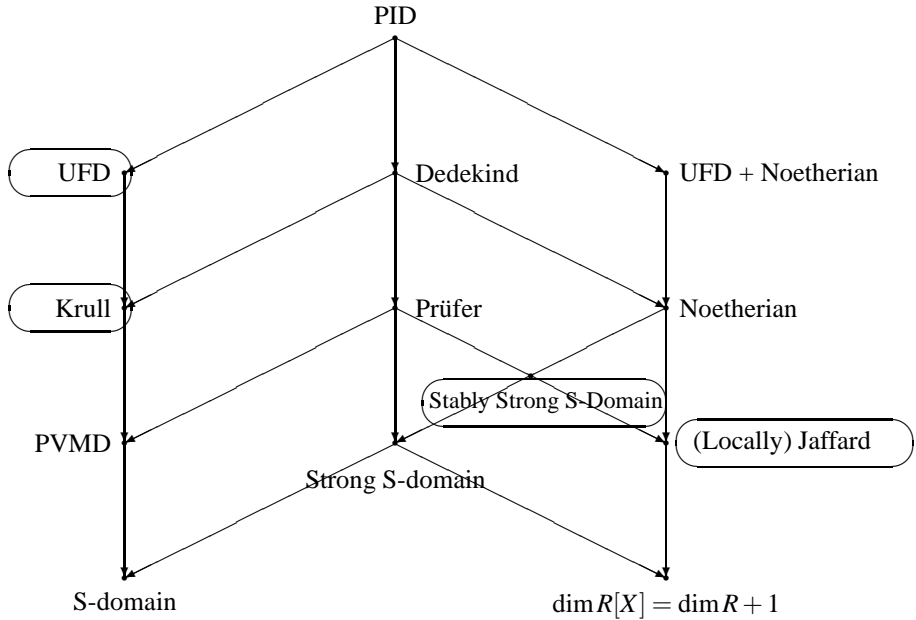
\begin{figure}[h!]
\[\setlength{\unitlength}{.9mm}
\begin{picture}(100,80)(0,0)
\put(40,80){\vector(0,-2){20}} \put(40,80){\vector(-2,-1){40}}
\put(40,80){\vector(2,-1){40}} \put(0,60){\vector(0,-2){20}}
\put(40,60){\vector(-2,-1){40}} \put(40,60){\vector(0,-2){20}}
\put(40,60){\vector(2,-1){40}} \put(80,60){\vector(0,-2){20}}
\put(80,20){\vector(0,-2){20}} \put(40,40){\vector(-2,-1){40}}
\put(40,40){\vector(2,-1){40}} \put(40,40){\vector(0,-2){20}}
\put(0,20){\vector(0,-2){20}} \put(80,40){\vector(0,-2){20}}
\put(80,40){\vector(-2,-1){40}} \put(0,40){\vector(0,-2){20}}
\put(40,20){\vector(-2,-1){40}} \put(40,20){\vector(2,-1){40}}

\put(-8,40){\oval(18,7)} \put(-8,60){\oval(18,7)}
\put(62,26){\oval(36,7)} \put(99,20){\oval(35,7)}

\put(40,80){\circle*{.7}} \put(40,82){\makebox(0,0)[b]{PID}}
\put(0,60){\circle*{.7}} \put(-2,60){\makebox(0,0)[r]{UFD}}
\put(0,40){\circle*{.7}} \put(-2,40){\makebox(0,0)[r]{Krull}}
\put(0,20){\circle*{.7}} \put(-2,20){\makebox(0,0)[r]{PVMD}}
\put(40,60){\circle*{.7}} \put(43,60){\makebox(0,0)[l]{Dedekind}}
\put(40,40){\circle*{.7}} \put(43,40){\makebox(0,0)[l]{Pr\"ufer}}
\put(40,20){\circle*{.7}} \put(40,16){\makebox(0,0)[t]{Strong
S-domain}}
\put(80,20){\circle*{.7}}\put(84,20){\makebox(0,0)[l]{(Locally)
Jaffard}}
\put(60,30){\circle*{.7}}\put(45,26){\makebox(0,0)[l]{\small Stably
Strong S-Domain}}
\put(80,0){\circle*{.7}}\put(80,-2){\makebox(0,0)[t]{$\dim R[X]
=\dim R + 1$}}
\put(0,0){\circle*{.7}}\put(0,-2){\makebox(0,0)[t]{S-domain}}
\put(80,40){\circle*{.7}} \put(82,40){\makebox(0,0)[l]{Noetherian}}
\put(80,60){\circle*{.7}} \put(82,60){\makebox(0,0)[l]{UFD +
Noetherian}}
\end{picture}\]\medskip

\caption{\tt Diagram of Implications} \label{D}
\end{figure}

Section~\ref{sec:2} of this paper provides some general results on the prime ideal structure of subalgebras $R$ of affine domains over a domain $A$, which actually turn out to equal the graded ring $\sum_{n\geq 0}a^{-n}{I_n}$ for some $a\in A$ and $a$-filtration $(I_n)_{n\geq 0}$ of $A$. In particular, we prove that $p\in \Spec(A)$ with $a\in p$ is the contraction of a prime ideal of $R$ if and only if $a^{-n}pI_{n}\cap A=p$ for each $n$. Moreover, if any one condition holds, then $p\supseteq I_{1}\supseteq I_{2}, \cdots$ (Corollary~\ref{sec:2.4}). Section~\ref{sec:3} examines closely the special construction of symbolic Rees algebra. One of the main results (Theorem~\ref{sec:3.1}) reveals the fact that three or more prime ideals of the symbolic Rees algebra may contract on the same prime ideal in the base ring. Also we show that if $A$ is a Noetherian domain and $p$ a maximal ideal of $A$, then the Rees algebra of $p$ is a stably  strong S-domain hence locally Jaffard (Theorem~\ref{sec:3.4}). Section~\ref{sec:4} investigates the dimension of graded rings associated with symbolic Rees algebras of prime ideals $p$ such that $A_{p}$ is a rank-one DVR and closes with an application related to Hochster's study of criteria that force the coincidence of the ordinary and symbolic powers of a prime ideal.
\end{section}

\begin{section}{The general context}\label{sec:2}

Recall that an affine domain over a ring $A$ is a finitely generated $A$-algebra that is a domain \cite[p. 127]{Na2}. In light of the developments described in \cite[Section 3]{BK2}, in order to investigate the prime ideal structure of subalgebras of affine domains over a Noetherian domain, we are reduced to those domains $R$ between a Noetherian domain $A$ and its localization $A[a^{-1}]$ for a nonzero element $a$ of $A$. For this purpose, we use the language of filtrations.

From \cite{At,Ma}, a filtration of a ring $A$ is a descending chain $(I_n)_n$ of ideals of $A$ such that $A=I_0\supseteq
I_1\supseteq\cdots \supseteq I_n\supseteq\cdots $ and
 $I_nI_m\subseteq I_{n+m}$ for all $n,m$. The associated graded ring of $A$ with respect to the filtration $(I_n)_n$
is given by $\gr(A):=\bigoplus_n \frac
{I_n}{I_{n+1}}$. The filtration $(I_n)_n$ is
said to be an $I$-filtration, for a given ideal $I$ of $A$, if
$II_n\subseteq I_{n+1}$ for each integer $n\geq 0$.

Let $A$ be a domain and $a$ a nonzero element of $A$. Let $(I_n)_n$ be an $a$-filtration of $A$ and $R:=A + a^{-1}{I_1} + a^{-2}{I_2} + \cdots  =\sum_{n\geq 0}a^{-n}{I_n}$. Clearly, $R$ is a domain,  which is an ascending union of the fractional ideals $(a^{-n}{I_n})_{n}$, such that $A\subseteq R\subseteq A[a^{-1}]$. The converse is also true as shown below.

\begin{lemma}\label{sec:2.1}
Let $A$ be a domain and $a\not=0\in A$. Then $R$ is a domain such that
$A\subseteq R\subseteq A[a^{-1}]$ if and only if $R:=A + a^{-1}{I_1} + a^{-2}{I_2} + \cdots  =\sum_{n\geq 0}a^{-n}{I_n}$ for some $a$-filtration $(I_n)_{n\geq 0}$ of $A$.
\end{lemma}

\begin{proof}
We only need to prove necessity. Let $I_n:=\{x\in A \mid {\frac
x{a^n}}\in R\}$ for each positive integer $n$. It is fairly easy to
see that $I_n$ is an ideal of $A$ for each integer $n$. Now,
let $x\in I_{n+1}$. Then $ {\frac x{a^{n+1}}}\in R$, so
that $a {\frac x{a^{n+1}}}=  {\frac
x{a^n}}\in R$. Thus $x\in I_n$. Also, observe that $aI_n\subseteq
I_{n+1}$ for each $n$. It follows that $(I_n)_n$ is an
$a$-filtration of $A$, as desired.
\end{proof}

\begin{lemma}\label{sec:2.2}
Let $A$ be a domain, $a\not=0\in A$, and $(I_n)_{n\geq 0}$ an
$a$-filtration of $A$. Let $R:=\sum_{n\geq 0}a^{-n}{I_n}$. Then the prime ideals of $R$ which don't contain $a$ are in one-to-one correspondence with the prime ideals of $A$ which don't contain $a$.
\end{lemma}

\begin{proof}
This follows from the fact that $S^{-1}A=S^{-1}R$, where $S$ is the multiplicatively closed subset of $A$ defined by $S:=\{a^{n} \mid n\in\N\}$. Moreover, if $P\in\Spec(R)$ with $a\notin P$ and $p:=P\cap A$, then
$$P=S^{-1}p\cap R=p[a^{-1}]\cap R=\displaystyle\sum_{n\geq 0}a^{-n}(p\cap I_n).$$
\end{proof}

The question which naturally arises is under what conditions a chain $q\subset p$ in $\Spec(A)$ with $a\in p\smallsetminus q$ lifts to a chain in $\Spec(R)$. This is handled by the main result of this section.

\begin{thm}\label{sec:2.3}
Let $A$ be a domain, $a\not=0\in A$, $(I_n)_{n\geq 0}$ an
$a$-filtration of $A$, and $R:=\sum_{n\geq 0}a^{-n}{I_n}$. Let $q\subset p\in\Spec(A)$ such that $a\in p\smallsetminus q$ and let $Q:=q[a^{-1}]\cap R$. Then the following assertions are equivalent:
\begin{enumerate}
\item There exists $P\in\Spec(R)$ such that $Q\subset P$ and $P\cap A=p$;
\item $a^{-n}(pI_{n} + q\cap I_{n})\cap A=p$, for each $n\geq0$.
\end{enumerate}
\end{thm}

\begin{proof} Recall first that $Q:=q[a^{-1}]\cap R=\sum_{n\geq 0}a^{-n}(q\cap I_n)$ is the unique prime ideal of $R$ lying over $q$ in $A$ (Lemma~\ref{sec:2.2}).

(1) $\Longrightarrow$ (2) Suppose  there exists $P\in\Spec(R)$ such that $Q\subset P$ and $P\cap A=p$. It is worth noting that $R_p= {\sum_{n\geq 0} a^{-n}{I_{n}A_{p}}}$ is associated with the $a$-filtration $(I_{n}A_{p})_n$ of $A_p$ and $Q_{p}= \sum_{n\geq 0} a^{-n}(qA_{p}\cap I_{n}A_{p})$ is the unique prime ideal of $R_{p}$ lying over $qA_{p}$ in $A_{p}$. Also $pR + Q= \sum_{n\geq 0} a^{-n}(pI_{n}+q\cap I_{n})$ and hence $pR_p + Q_{p}= \sum_{n\geq 0} a^{-n}(pI_{n}A_{p}+qA_{p}\cap I_{n}A_{p})=\bigcup_{n\geq 0} a^{-n}(pI_{n}A_{p}+qA_{p}\cap I_{n}A_{p})$, an ascending union of fractional ideals of $R_p$. Now $pR_p + Q_{p}$ is a proper ideal of $R_p$. Therefore, for each $n$, $a^n\not\in pI_{n}A_{p}+qA_{p}\cap I_{n}A_{p}$. Hence $sa^n\not\in pI_{n}+q\cap I_{n}$ for every $s\in A\setminus p$, whence $(a^{-n}(pI_{n} + q\cap I_{n})\cap A)\cap (A\smallsetminus p)=\emptyset$. It follows that $a^{-n}(pI_{n} + q\cap I_{n})\cap A= p$ for each $n\geq0$.

(2) $\Longrightarrow$ (1) Suppose $a^{-n}(pI_{n} + q\cap I_{n})\cap A=p$ for each $n\geq0$.  Then
$(pR+Q)\cap A= \bigcup_{n\geq0}\big(a^{-n}(pI_{n} + q\cap I_{n})\cap A\big)=p$. Therefore \cite[Proposition 3.16]{At} applied to the ring homomorphism $\frac{A}{q} \hookrightarrow \frac{R}{Q}$ leads to the conclusion.
\end{proof}

The special case where $q=0$ yields a necessary and sufficient condition for a prime ideal of $A$ containing $a$ to lift to a prime ideal of $R$.

\begin{corollary}\label{sec:2.4}
Let $A$ be a domain, $a\not=0\in A$, $(I_n)_{n\geq 0}$ an
$a$-filtration of $A$, and $R:=\sum_{n\geq 0}a^{-n}{I_n}$. Let $p\in\Spec(A)$ such that $a\in p$. Then $p$ is the contraction of a prime ideal of $R$ if and only if $a^{-n}pI_{n}\cap A=p$ for each $n$. Moreover, if any one condition holds, then $p\supseteq I_{1}$.
\end{corollary}

\begin{proof}
The equivalence is ensured by the above theorem with $q=0$. Moreover, $a\in p$ yields $I_{1} \subseteq a^{-1}pI_{1}\cap A=p$, as desired.
\end{proof}

Now it is legitimate to ask whether there may exist a chain of prime ideals
of $R$ of length $\geq2$ lying over a given prime ideal $p$ of $A$ containing $a$. Ahead, Corollary~\ref{sec:3.2} gives an affirmative answer to this question.

\end{section}

\begin{section}{The case of symbolic Rees algebras}\label{sec:3}

Here we will focus on the special case of symbolic Rees algebras. In 1958, Rees constructed in \cite{Re} a first counter-example to the
Zariski-Hilbert problem (initially posed at the Second International Congress of Mathematicians at Paris in 1900). His construction gave rise to
(what is now called) Rees algebras. Since then, these special graded algebras have been capturing the interest of many mathematicians, particularly
in the fields of commutative algebra and algebraic geometry.

Let $A$ be a domain, $t$ an indeterminate over $A$, and $p\in$ Spec$(A)$. For
each $n\in \Z$, set $p^{(n)}:=p^nA_p\cap A$, the $n^{th}$
symbolic power of $p$, with $p^{(n)}=A$ for each $n\leq 0$. Notice that $p=p^{(1)}$ and $p^{n}\subseteq p^{(n)}$ for all $n\geq2$. We recall the following definitions:\\
$\bullet$ $\bigoplus_{n\in \Z}p^nt^n=A[t^{-1},pt,\cdots ,p^nt^n,\cdots ]$ is the Rees algebra of $p$.\\
$\bullet$ $\bigoplus_{n\in \Z}p^{(n)}t^n=A[t^{-1},p^{(1)}t,\cdots ,p^{(n)}t^n,\cdots ]$ is the symbolic Rees algebra of $p$.

In 1970, based on Rees' work, Eakin and Heinzer constructed in \cite{EH} the first example of a $3$-dimensional non-Noetherian Krull domain. It arose
as a symbolic Rees algebra.  This enhances our interest for these constructions. In this section, we wish to
push further the analysis of the prime deal structure of symbolic Rees algebras. Precisely, we plan to investigates the lifting of prime ideals of
$A[t^{-1}]$ in the symbolic Rees algebra $R$. We prove that any prime ideal of $A[t^{-1}]$ lifts to a prime ideal in $R$. We also examine the length of chains of prime ideals of $R$ lying over a prime ideal of $A[t^{-1}]$.

Let us fix the notation for the rest of this section. Let $A$ be a domain and $t$ an indeterminate over $A$. Let $p\in\Spec(A)$ and let $$R:=A[t^{-1},pt,p^{(2)}t,\cdots ,p^{(n)}t^n,\cdots ]$$ be the symbolic Rees algebra of $p$. Consider the $t^{-1}$-filtration $(I_n)_{n\geq0}$ of $A[t^{-1}]$, where $I_{0}=A[t^{-1}]$,  $I_{1}=p[t^{-1}]+t^{-1}A[t^{-1}]$, and for $n\geq2$
$$I_n:=p^{(n)}[t^{-1}]+t^{-1}p^{(n-1)}[t^{-1}]+\cdots +t^{-(n-1)}p[t^{-1}]+t^{-n}A[t^{-1}].$$ One can easily check that
$$A[t^{-1}]\subseteq R\subseteq A[t^{-1},t] \mbox { and }R={\sum_{n\geq0}I_nt^n=\bigcup_{n\geq0}I_nt^n.}$$
Finally, for $q\supseteq p$ in $\Spec(A)$, set
$$\G(A):=\bigoplus_{n\geq0} \frac{p^{(n)}}{p^{(n+1)}} \mbox { and } \G(A_{q}):=\bigoplus_{n\geq0} \frac{pA_{q}^{(n)}}{pA_{q}^{(n+1)}}.$$

The first result examines the transfer of the Jaffard property.

\begin{proposition}\label{sec:3.0}
Assume $A$ to be a Jaffard domain. Then $R$ is a Jaffard domain with $\dim(R)=1+\dim(A)$.
\end{proposition}

\begin{proof}
Notice that $A[t^{-1}]\subseteq R\subseteq A[t^{-1},t]$. By \cite[Lemma 1.15]{ABDFK}, $\dim_{v}(R)=\dim_{v}(A[t^{-1}])$. On the other hand, the equality $R[t]=A[t^{-1},t]$ combined with \cite[Proposition 1.14]{ABDFK} yields $\dim(A[t^{-1}])\leq \dim(R)$. Now $A$ is Jaffard and then so is $A[t^{-1}]$ \cite[Proposition 1.2]{ABDFK}. Consequently, $\dim(R)=\dim_{v}(R)=1+\dim(A)$, as desired.
\end{proof}

Next, we investigate the prime ideals of $A[t^{-1}]$ that lift in the symbolic Rees algebra $R$. In view of Lemma~\ref{sec:2.2},
 one has to narrow the focus to the prime ideals which contain $t^{-1}$. Moreover, by Corollary~\ref{sec:2.4}, these primes
must necessarily contain $I_{1}=(p,t^{-1})$. Consequently, we reduce the study to the prime ideals of $A[t^{-1}]$ of the form $(q,t^{-1})$ where $q\supseteq p\in\Spec(A)$.

\begin{thm}\label{sec:3.1}
Let $q$ be a prime ideal of $A$ containing $p$. Then the following lattice isomorphisms hold:
\begin{enumerate}
\item $\big\{Q\in\Spec(R) \mid Q\cap A[t^{-1}] =(q,t^{-1})\big\}\simeq\Spec\left(\dfrac{\G(A_{q})}{\frac{qA_{q}}{pA_{q}}\G(A_{q})}\right)$.
\item $\big\{P\in\Spec(R) \mid P\cap A[t^{-1}] =(p,t^{-1})\big\}\simeq\Spec\big(\G(A_{p})\big)$.
\end{enumerate}
\end{thm}

\begin{proof}
(1) Let $\G(q):=\displaystyle {\frac Aq\oplus \frac p{p^{(2)}+qp}\oplus \frac
{p^{(2)}}{p^{(3)}+qp^{(2)}}\oplus \cdots }$. We claim that $${\dfrac{R}{(q,t^{-1})R}\cong \G(q)\cong \dfrac{\G(A)}{(q/p)\G(A)}}.$$ Indeed, notice the following:
$$\left \{
\begin{array}{lll}
R&=&A[t^{-1}]\oplus pt\oplus p^{(2)}t^2\oplus \cdots \oplus p^{(n)}t^n\oplus \cdots \\
&&\\
t^{-1}R&=&(p,t^{-1})\oplus p^{(2)}t\oplus\cdots \oplus p^{(n+1)}t^n\oplus\cdots \\
&&\\
qR&=&q[t^{-1}]\oplus qpt\oplus qp^{(2)}t^2\oplus\cdots \oplus qp^{(n)}t^n\oplus\cdots \\
&&\\
(q,t^{-1})R&=&(q,t^{-1})\oplus (p^{(2)}+qp)t\oplus\cdots \oplus (p^{(n+1)}+qp^{(n)})t^n\oplus\cdots
\end{array}
\right.$$
Then it is easily seen that ${\dfrac
R{t^{-1}R}\cong \G(A)}$ and ${\dfrac R{(q,t^{-1})R}\cong
\G(q)}$. Moreover,

$\begin{array}{lll} \G(q)&\cong& \displaystyle {\frac
{A/p}{q/p}\oplus \frac {p/p^{(2)}}{(q/p)(p/p^{(2)})}\oplus\cdots \oplus
\frac {p^{(n)}/p^{(n+1)}}{(q/p)(p^{(n)}/p^{(n+1)})}}\oplus\cdots \\
&\cong&\dfrac{\G(A)}{(q/p)\G(A)}.
\end{array}$

Now, observe that $R_q=A_q[t^{-1},pA_q^{(1)}t,\cdots ,pA_q^{(n)}t^n,\cdots ]$ is the symbolic Rees algebra of $pA_q$. This is due to the fact that
$p^{(n)}A_q=p^nA_p\cap A_q=pA_q^{(n)}$ for each $n\geq0$. We obtain
$$\frac {R_q}{(q,t^{-1})R_q}=\frac{R_q}{(qA_{q},t^{-1})R_q}\cong \dfrac{\G(A_{q})}{\frac{qA_{q}}{pA_{q}}\G(A_{q})}.$$
Hence the set of prime ideals of $R$
lying over $(q,t^{-1})$ in $A[t^{-1}]$ is lattice isomorphic to the spectrum of $\frac{\G(A_{q})}{(qA_{q}/pA_{q})\G(A_{q})}$.

(2) Take $q:=p$ in (1), completing the proof of the theorem.
\end{proof}

We deduce the following result in the Noetherian case. It shows, in particular, that  there may exist a chain of prime ideals
of $R$ of length $\geq2$ lying over $(p,t^{-1})$ in $A[t^{-1}]$.

\begin{corollary}\label{sec:3.2}
Assume that $A$ is Noetherian and let $n:=\htt(p)$. Then the set $\big\{P\in\Spec(R) \mid P\cap A[t^{-1}] =(p,t^{-1})\big\}$ is lattice isomorphic to the spectrum of an n-dimensional finitely generated algebra over the field $\dfrac{A_{p}}{pA_{p}}$.
\end{corollary}

\begin{proof}
Let $y_1,\cdots ,y_r\in A$ such that $pA_p=(y_1,\cdots ,y_r)A_p$ and let $e_1,\cdots ,e_r$ denote their respective images in $\dfrac{pA_p}{p^2A_p}$. By Theorem~\ref{sec:3.1}, $\dfrac {R_p}{(p,t^{-1})R_p}\cong \G(A_{p})$. Note that $R_{p}$ coincides with the Rees algebra of $pA_{p}$ (since $p^{n}A_{p}=p^{(n)}A_{p}$ for all $n\geq1$). It follows that $$\G(A_{p})=\gr(A_p)=\dfrac{A_{p}}{pA_{p}}[e_1,\cdots ,e_r]\ \hbox{(cf. \cite[p. 93]{Ma})}.$$ On the other hand, by \cite[Theorem 15.7]{Ma}, $\dim(\gr(A_p))=\dim(A_{p})=n$, completing the proof.
\end{proof}

Notice at this point that $t^{-1}R=(p,t^{-1})R$ (see the proof of Theorem~\ref{sec:3.1}). This translates into the fact that prime ideals of $R$ containing $t^{-1}$ contain necessarily $p[t^{-1}]$ (stated in Corollary~\ref{sec:2.4}). Given a prime ideal $q$ of $A$, we next exhibit particular prime ideals of $R$ that lie over $q$.

\begin{proposition}\label{sec:3.3}
Let $q\in\Spec(A)$. The following hold:
\begin{enumerate}
\item Assume $p\subseteq q$. Then $Q:=(q,t^{-1})\oplus pt\oplus p^{(2)}t^2\oplus \cdots $ is a prime ideal of $R$ lying over $(q,t^{-1})$ in $A[t^{-1}]$ and $P$ is maximal with this property.

\item $q[t^{-1},t]\cap R=q[t^{-1}]\oplus (p\cap q)t\oplus (p^{(2)}\cap q)t^2\oplus\cdots $ is the unique prime ideal of $R$ lying over $q[t^{-1}]$ in $A[t^{-1}]$.
\end{enumerate}
\end{proposition}

\begin{proof}
(1) Assume $p\subseteq q$. It is easily seen that $\dfrac RQ\cong \dfrac Aq$. It follows that $Q$ is a
prime ideal of $R$ and $Q\cap A[t^{-1}]=(q,t^{-1})$. Now $R_q$ is the symbolic Rees algebra of $pA_q$ with
$\dfrac{R_q}{PR_q}\cong \dfrac{A_q}{qA_q}$, a field. Therefore $Q$ is maximal among the prime ideals of $R$ lying over $(q,t^{-1})$.

(2) By Lemma~\ref{sec:2.2}, the unique prime ideal of $R$ lying over $q[t^{-1}]$ is $q[t^{-1},t]\cap R$. Further observe that
$q[t^{-1},t]=q[t^{-1}]\oplus qt\oplus qt^{2}\oplus \cdots $. So that $q[t^{-1},t]\cap R=q[t^{-1}]\oplus (p\cap q)t\oplus (p^{(2)}\cap q)t^2\oplus\cdots $, as claimed.
\end{proof}

\begin{thm}\label{sec:3.4}
Assume that $A$ is Noetherian and $p\in\Max(A)$. Then $R$ is a stably  strong
S-domain (hence locally Jaffard).
\end{thm}

\begin{proof}
Let $T:=A[t^{-1},pt,p^2t^2,\cdots ,p^nt^n,\cdots ]$ be the Rees algebra of $p$.
Let $n$ be a positive integer. Consider the natural injective ring homomorphism:\\ $T[X_1,\cdots ,X_n]\hookrightarrow R[X_1,\cdots ,X_n]$. This induces the following map $$f:\Spec(R[X_1,\cdots ,X_n])\longrightarrow\Spec(T[X_1,\cdots ,X_n])$$ defined by $f(P)=P\cap T[X_1,\cdots ,X_n]$. We claim that $f$ is an order-preserving bijection. Indeed, let $Q$ be a prime ideal of $T[X_1,\cdots ,X_n]$. If $t^{-1}\not\in Q$, then $Q$ survives in
$$A[t^{-1},t,X_1,\cdots ,X_n]=R[t,X_1,\cdots ,X_n]=T[t,X_1,\cdots ,X_n].$$ Therefore $P:=QA[t^{-1},t,X_1,\cdots ,X_n]\cap R[X_1,\cdots ,X_n]$. Hence $P$ is the unique prime ideal of $R[X_1,\cdots ,X_n]$ such that $f(P)=Q$. Now, let $t^{-1}\in Q$. Then $(p,t^{-1})\subseteq Q\cap A[t^{-1}]$ by Corollary~\ref{sec:2.4}, whence $p=Q\cap A$ as $p$ is maximal in $A$. Moreover recall that $R_p=T_p$. Therefore $Q$ survives in $T_p[X_1,\cdots ,X_n]=R_p[X_1,\cdots ,X_n]$ and hence $P:=QR_p[X_1,\cdots ,X_n]\cap R[X_1,\cdots ,X_n]$ is the unique prime ideal of $R[X_1,\cdots ,X_n]$ such that $f(P)=Q$. It follows that $f$ is bijective. Obviously, it also preserves the inclusion order. Now assume $p=(a_1,\cdots ,a_r)$. One can easily check that $T=A[t^{-1},a_{1}t,\cdots ,a_{r}t]$, so that $T$ is Noetherian and thus a stably  strong S-domain. It follows that $T[X_1,\cdots ,X_n]$ is a strong S-domain and so is $R[X_1,\cdots ,X_n]$, as desired.
\end{proof}

\begin{remark}
It is worth noting that the proof of the above theorem is still valid if we weaken the assumption ``$A$ is Noetherian" to ``$A$ is a stably  strong S-domain and $p$ is finitely generated" since the concept of strong S-domain is stable under quotient ring.
\end{remark}
\end{section}

\begin{section}{Associated graded rings and applications}\label{sec:4}

This section investigates the dimension theory of graded rings associated with special symbolic Rees algebras. Recall that the
Krull dimension of the graded ring associated with the (ordinary) Rees algebra of an ideal $I$ of a Noetherian domain $A$ is given by the formula (cf. \cite[Exercise 13.8]{E}):
$$\dim(\gr_{I}(A))=\max\{ht(q)\mid q\in\Spec(A)\mbox{ and }I\subseteq q\}.$$

Let us fix the notation for this section. Throughout $A$ will denote a Noetherian domain and $p$ a prime ideal of $A$ such that $A_p$ is a rank-one DVR. Thus, any height-one prime ideal of an integrally closed Noetherian domain falls within the scope of this study.
Let $R$, $\G(A)$, and $\gr(A)$ denote the symbolic Rees algebra of $p$, the associated graded ring of $A$ with respect to the filtration $(p^{(n)})_n$, and the associated graded ring of $p$, respectively. That is,
$$R:=A[t^{-1},pt,p^{(2)}t^2,\cdots ,p^{(n)}t^n,\cdots ],$$
$$\G(A):=\bigoplus_{n\geq0} \frac{p^{(n)}}{p^{(n+1)}},$$
$$\gr(A):=\bigoplus_{n\geq0} \frac{p^{n}}{p^{n+1}}.$$
Finally,  let $u\in p$ such that $pA_{p}=uA_{p}$ and $v:=\overline{u}$ be the image of $u$ in $p/p^{(2)}$.

\begin{lemma}\label{sec:4.1}
For each $n\geq0$, let $E_n:=(A:_{A_p}u^n)=\{x\in A_p\mid xu^n\in A\}$ and $F_n:=\overline {E_n}$ be the image of $E_n$ in $K:=A_p/pA_p$. Then:
\begin{enumerate}
\item $(E_n)_{n\geq0}$ is an ascending sequence of fractional ideals of $A$ such that $A\subseteq E_n\subseteq A_p$
and $p^{(n)}=E_nu^n$, for each $n$.

\item $(F_n)_{n\geq0}$ is an ascending sequence of fractional ideals of $\dfrac{A}{p}$ such that $\dfrac{A}{p}\subseteq F_n\subseteq K$
and $\dfrac{p^{(n)}}{p^{(n+1)}}=F_nv^n$, for each $n$.

\item $\G(A)=\bigoplus_{n\geq 0}F_nv^n$.
\end{enumerate}
\end{lemma}

\begin{proof}  Clearly, $(E_n)_{n}$ is an ascending sequence of fractional ideals of $A$. Fix $n\geq 0$. We have $x\in p^{(n)}$ if and only if $x\in u^nA_p$ and $x\in A$ if and only if there exists $y\in A_p$ such that $x=yu^n\in A$ if and only if $x\in E_nu^n$. This proves (1). Assertion (2) is a consequence of (1) and the proof is left to the reader. Also (3) is trivial from (2).
\end{proof}

Next, we announce the main result of this section.

\begin{thm}\label{sec:4.2}
Let $D:=\bigcup_{n\geq 0}F_n$ and $X$ an indeterminate over $D$. Then:
\begin{enumerate}
\item $\G(A)$ is a Jaffard domain and $\dim(\G(A))=1+\dim(A/p)$.
\item $\dim(\G(A)/v\G(A))=\dim(A/p)$ and  $\dim(\G(A)[v^{-1}])=\dim(D[X])$.
\end{enumerate}
\end{thm}

\begin{proof}
We first prove the following claims.

\begin{claim}
$D$ is an overring of $\dfrac Ap$ and $v$ is transcendental over $D$.
\end{claim}

It is fairly easy to see that $F_nF_m\subseteq F_{n+m}$ for any $n$ and $m$. It follows that $D$ is an overring of $\dfrac Ap$ contained in $K$. Let
$P=b_0+b_1X+\cdots +b_nX^n\in \displaystyle {\frac Ap}[X]$ such that
$P(v)=0=b_0+b_1v+\cdots +b_nv^n$. Let $i\in\{0,1,\cdots, n\}$. Since $b_iv^i\in F_{i}v^i$, $b_iv^i=0$ by Lemma~\ref{sec:4.1}.
So $b_i=\overline {a_i}$ (mod $pA_{p}$), for some $a_i\in E_i$, and $a_iu^i\in p^{i}A_{p}\cap A=p^{(i)}$. Therefore $\overline{a_iu^i}=0$ in $\dfrac{p^{(i)}}{p^{(i+1)}}$, that is, $a_iu^i\in p^{(i+1)}$. Hence $a_i\in pAp$, whence $b_i=0$. Consequently, $P=0$, proving that $v$ is transcendental over $D$.

\begin{claim}
$\displaystyle {\frac Ap}[v]\subseteq \G(A)\subseteq D[v]$.
\end{claim}

This follows from the facts that $\dfrac Ap\subseteq F_n\subseteq D$ for each $n\geq 0$ and $\G(A)=\bigoplus_{n\geq 0}F_nv^n$ by Lemma~\ref{sec:4.1}.

\begin{claim}
$S^{-1}\G(A)=D[v,v^{-1}]$, where $S:=\{v^n\mid n\geq 0\}$.
\end{claim}

Clearly, $S^{-1}\G(A)\subseteq D[v,v^{-1}]$. Also note that
$D\subseteq S^{-1}\G(A)$ since $F_n=(F_nv^n)v^{-n}\subseteq S^{-1}\G(A)$ for
each positive integer $n$. Hence $D[v,v^{-1}]\subseteq
S^{-1}\G(A)\subseteq D[v,v^{-1}]$ establishing the desired equality.

(1) In view of Claim 2, we get dim$_v(\G(A))\leq 1+$dim$_v(A/p)=1+$dim$(A/p)$. On the other hand,
notice that, for each prime ideal $q$ of $A$ containing $p$, the ideal
$Q:=\displaystyle {\frac qp\oplus \frac p{p^{(2)}}\oplus \frac
{p^{(2)}}{p^{(3)}}\oplus}\displaystyle {\cdots}\in\Spec(\G(A))$ with $\dfrac{\G(A)}{Q}\cong \dfrac{A}{q}$. So
dim$(\G(A))\geq 1+$dim$(A/p)$ as $\G(A)$ is a domain. Thus
dim$(\G(A))=$ dim$_v(\G(A))=1+$dim$(A/p)$.

(2) First notice that $$\htt(v\G(A))+\dim(\G(A)/v\G(A))\leq\dim(\G(A))=1+\dim(A/p).$$ Then
dim$(\G(A)/v\G(A))\leq$ dim$(A/p)$. Consider the prime ideal of $\G(A)$ given by $P:=F_1v\oplus F_2v^2\oplus\cdots$ and
$p\subset p_1\subset p_2\subset \cdots \subset p_h\in\Spec(A)$ with $h:=\dim(A/p)$. We get the following chain of prime
ideals of $\G(A)$ containing the ideal $v\G(A)$
$$v\G(A)\subset P\subset \displaystyle {\frac{p_1}p\oplus P\subset \frac {p_2}p\oplus P\subset\cdots  \subset \frac{p_h}p\oplus P}.$$
It follows that dim$(\G(A)/v\G(A))=$ dim$(A/p)$. Moreover, Claims 1 and 3 yield  $$\dim(S^{-1}\G(A))=\dim(D[v,v^{-1}])=\dim(D[X])$$ completing the proof of the theorem.
\end{proof}

Let $B:=\bigcup_{n\geq 0}E_n=\bigcup_{n\geq 0}u^{-n}p^{(n)}$. Notice that $B$ is an overring of $A$ contained in $A_p$ and $\overline B:=\frac B{pA_p\cap B}=D$. The next result investigates some properties of $B$ and its relation with $D$, in view of the fact that an essential part of the spectrum of $\G(A)$ (and hence that of $R$) is strongly linked to $D$.

\begin{proposition}\label{sec:4.3}
Let $B:=\bigcup_{n\geq 0}u^{-n}p^{(n)}$. Then:
\begin{enumerate}

\item $pB=uB$ is a height-one prime ideal of $B$ and it is the unique prime ideal of $B$ lying over $p$ in $A$.

\item $\displaystyle {\frac B{uB}=D}$.

\item $B=A_p\cap A[u^{-1}]$. Then, if $A$ is a Krull domain, so is $B$.

\item $B$ is locally Jaffard if and only if so is $D$.
\end{enumerate}
\end{proposition}

\begin{proof}  (1) Let $z\in pA_p\cap B$. Then ${z=\frac as=\frac x{u^n}}$ for some positive integer $n$, with $x\in p^{(n)}$, $a\in p$ and $s\in
A\smallsetminus p$. Then $sx=au^n\in p^{n+1}$ which means that $x\in p^{(n+1)}$. Hence $z={u\frac x{u^{n+1}}\in u\frac{p^{(n+1)}}{u^{n+1}}}\subseteq uB$. Therefore $pA_p\cap B=uB=pB$, whence $pB\in\Spec(B)$ with $pB\cap A=p$. Moreover, observe that $B_{p}:=\bigcup_{n\geq 0}u^{-n}p^{(n)}A_{p}=\bigcup_{n\geq 0}u^{-n}p^{n}A_{p}=A_{p}$ is a rank-one DVR. Then $\htt(pB)=\htt(pB_{p})=\htt(pA_{p})=1$ and $pB$ is the unique prime ideal of $B$ lying over $p$ in $A$.

(2) It is straightforward from (1) and the fact that $\overline
B:=\displaystyle {\frac B{pA_p\cap B}=D}$.

(3) It is clear that $B\subseteq A_p\cap A[u^{-1}]$. Let $z\in
A_p\cap A[u^{-1}]$. Then $z={\frac x{u^n}=\frac as}$
for some positive integer $n$, and $x,a\in A$ and $s\in A\smallsetminus p$. So $xs=au^n\in p^n$.
Hence $x\in p^{(n)}$ which means that $z\in {\frac
{p^{(n)}}{u^n}}\subseteq B$. Then the desired equality holds.

(4) Applying (1), one can check that the following diagram is cartesian:
$$\begin{array}{lll}
B&\longrightarrow&D\\
\downarrow&&\downarrow\\
A_p&\longrightarrow& K,
\end{array}
$$
which allows the transfer of the locally Jaffard property between $B$ and $D$ (recall that $A_{p}$ is a rank-one DVR).
\end{proof}

In \cite{H}, Hochster investigated when the symbolic power $p^{(n)}$
of a given prime ideal $p$ of a Noetherian domain $A$ coincides with
the ordinary power $p^n$ for any $n\geq0$. His main
theorem gives sufficient conditions guaranteeing this equality.
Applying this theorem, he proves that the Cohen-Macaulayness of
$A/p$ has nothing to do with the coincidence of
the symbolic and ordinary powers, by providing a polynomial ring in four
indeterminates $A$ such that $A/p$ is not
Cohen-Macaulay while $p^{(n)}=p^n$ for any positive integer $n$.
However, there are very few examples in the literature of Noetherian
domains $A$ for which there exists a prime ideal $p$ such that
$p^{(n)}\neq p^n$ for some positive integer. We cite here Northcott's Example \cite[Example 3, p. 29]{No} in which $p$
is the defining ideal of a curve (so that its residue class ring is
Cohen-Macaulay) while $p^{(2)}\neq p^2$. There is no new example in
Hochster's paper.

From Theorem~\ref{sec:4.2} we deduce a necessary condition for the symbolic and ordinary powers to coincide for a height-one prime ideal
$p$ of a Noetherian domain $A$. This will allow us to provide a bunch of original and new examples of Noetherian domains for which there exists a prime ideal $p$ such that $p^{(n)}\neq p^n$ for some positive integer $n$.

\begin{corollary}\label{sec:4.4}
 Let $A$ be a local Noetherian domain
and $p$ a prime ideal of $A$ such that $A_p$ is a rank-one DVR. Then:
$$p^{(n)}=p^n, \forall n\geq0 \Longrightarrow \dim(A) = 1 + \dim(A/p).$$
\end{corollary}

\begin{proof} Assume $p^{(n)}=p^n, \forall n\geq0$. Then $\G(A)=\gr(A)$. So a combination of \cite[Theorem 15.7]{Ma} and Theorem~\ref{sec:4.2} leads to the conclusion.
\end{proof}

\begin{corollary}\label{sec:4.5}
Let $A$ be an integrally closed local Noetherian domain which is not catenarian. Then there exists a prime ideal $p$ of $A$ such that
$p^{(n)}\neq p^n$ for some positive integer $n$.
\end{corollary}

\begin{proof}  Let $\frak{m}$ denote the maximal ideal of $A$. Since $A$ is not catenarian, there exists a height-one prime
ideal $p\subsetneqq \frak{m}$ of $A$ such that $$1+\dim(A/p)\lneqq ht(\frak{m})=\dim(A).$$ By Corollary~\ref{sec:4.4}, there exists $n\geq2$ such that $p^{(n)}\neq p^n$.
\end{proof}

Next, we exhibit an explicit example of a local Noetherian domain $A$ containing a prime ideal $p$ such that $p^{(n)}\neq p^n$ for some positive integer $n$. For this purpose, we'll use Nagata's well-known example of a Noetherian domain which is catenarian but not universally  catenarian \cite{Na2}.

\begin{example}\label{sec:4.6}
Let $k$ be a field and $X,Y,Z,t$ be indeterminates over $k$. Consider the
$k$-algebra homomorphism $\varphi:k[X,Y]\rightarrow k[[t]]$ defined by $\varphi (X)=t$ and $\varphi (Y)=s:={\sum_{n\geq
1}}t^{n!}$. Since $s$ is known to be transcendental over $k(t)$,
$\varphi$ is injective. This induces an embedding $\overline
{\varphi}:k(X,Y)\rightarrow k((t))$ of fields. So $B_1:=\overline
{\varphi}^{-1}(k[[t]])$ is a rank-one discrete valuation
overring of $k[X,Y]$ of the form $B_1=k+XB_1$. Let
$B_2:=k[X,Y]_{(X-1,Y)}$ and $B:=B_1\cap B_2$. Then $\Max(B)=\{M,N\}$ with $M=XB_1\cap B$ and
$N=(X-1,Y)B_2\cap B$, and $B$ is Noetherian. Let $C:=k+\frak{m}$ with
$\frak{m}:=M\cap N$. It turns out that $C$ is a 2-dimensional Noetherian domain such that the polynomial ring $C[Z]$ is not catenarian. So there is an upper $Q$ to $\frak{m}$ which contains an upper $P$ to
zero such that the chain $(0)\subsetneqq P\subsetneqq Q$ is saturated
with $\htt(Q)=3$. Now, let $A:=C[Z]_Q$ and $p:=PC[Z]_Q$. Then $A$ is a local Noetherian domain and $A_{p}\cong C[Z]_{P}$ is a rank-one DVR with $1+\dim(A/p)=2\lneqq \dim(A)=3$. Consequently, by Corollary~\ref{sec:4.4}, there exists $n\geq2$ such that $p^{(n)}\not=p^{n}$. \qed
\end{example}
\end{section}


\end{document}